\documentclass[12pt,twoside]{article}
\pdfoutput=1

    \usepackage{graphicx}
    \usepackage{styleset}
    \usepackage{macros}
    \usepackage[colorinlistoftodos]{todonotes}
    \usepackage{enumitem}
    \let\subsubsection\subparagraph

    \title  {The Akbulut cork is not universal}

    \author{Roberto Ladu}

		\date{\vspace{-2ex}}
\newcommand{\emb}{\mathbb{e}}
\newcommand{\interior}[1]{\mathrm{int}({#1})}
\newcommand{\FamilyVanishing}{\mathcal{C}_{\Delta = 0}}
\newcommand{\FamilyVanishingMir}{\overline{\mathcal{C}}_{\Delta = 0}}
\begin{document}

\maketitle            

\begin{abstract} We exhibit infinitely many exotic pairs of simply-connected, closed  $4$-manifolds not related by any cork of the infinite family $W_n$ constructed by Akbulut and Yasui whose first member is the Akbulut cork. In particular, the Akbulut cork is not universal.
Moreover we show that, in the setting of manifolds with boundary, there are no $\partial$-universal corks, i.e. there does not exist
a cork which relates any exotic pair of simply-connected $4$-manifolds with boundary.
\end{abstract}

%\tableofcontents

\section{Introduction}
We work in the category of smooth, compact, oriented manifolds with boundary.
The homeomorphism and diffeomorphism relations will be denoted respectively by $\approx$ and $\cong$, see Section~\ref{sec:Conventions} for more on our conventions.
An (abstract) cork is a pair $(C,f)$ where $C$ is a contractible  $4$-manifold and $f:\partial C\to \partial C$ is an involution that
does not extend to a diffeomorphism of $C$. The first cork $(W_1, f_{W_1})$ was discovered by Akbulut \cite{Akbulut91}, and fits in a family of corks  $\{(W_n,f_{W_n})\}_{n\in \N}$ \cite{AkbulutYasui09}.

A pair of $4$-manifolds $(X_0,X_1)$ is \emph{related by a cork} $(C,f)$ if there exists an embedding $\emb: C\to X_0$ such that
$X_0\setminus \interior{\emb(C)} \bigcup_{\emb\circ f} C \cong X_1$.
It is known that any exotic pair $(X_0, X_1)$ i.e. $X_0\approx X_1$, $X_0\not\cong X_1$, of simply-connected, closed $4$-manifolds
is related by some cork \cite{CHFS, Matveyev}.  In  \cite{AkbulutYasui09} it was posed the question, therein attributed to Akbulut,  whether the cork $(W_1, f_{W_1})$ is \emph{universal}, i.e. relates any exotic, simply-connected, closed  pair.
As noted in \cite{AKMREquivariantCorks}, it is unlikely though conceivable that there exists a universal cork. Therefore we recast this question in the form of a conjecture:
\begin{conjecture}[Universal cork] There is no universal cork, i.e. a cork $(U,f_U)$ relating any exotic pair of  simply-connected, closed $4$-manifolds.
\end{conjecture}
In this paper we prove this conjecture for several  corks  including the corks $(W_n,f_{W_n})$ and their mirrors  i.e. the cork obtained by reversing the orientation of $W_n$.
In particular, the case $n=1$ answers Akbulut's question from \cite{AkbulutYasui09}.

The difference element of a cork \cite{Ladu} is an invariant of corks taking values  in the Floer homology of the cork boundary (see Section~\ref{sec:DifferenceElement}). Denote by $\FamilyVanishing$ the set of corks with vanishing difference element and 
by $\FamilyVanishingMir$ the set of corks $(C,f)$ such that $(-C,f)\in \FamilyVanishing$. We prove the following theorem:
\begin{theorem} \label{thm:DeltaZeroNotUniv} There exist infinitely many exotic pairs $(X_0,X_1)$ of simply-connected, closed $4$-manifolds,
such that  \begin{enumerate}
					\item \label{itema} $(X_0,X_1)$ and $(-X_0, -X_1)$ are not  sequentially related  by $\FamilyVanishing$, 
					\item \label{itemb} $(X_0,X_1)$ and $(-X_0, -X_1)$ are not  sequentially related  by $\FamilyVanishingMir$,
				\end{enumerate}
				where $(X_0,X_1)$ is sequentially related by a set of corks $\mathcal S$ if $X_1 $ can be obtained from  $X_0$ with a \emph{sequence} of cork twists using corks belonging to $\mathcal S$ (see Def.~\ref{Def:sequentiallyRelated}).
\end{theorem}

We remark that the two items in the theorem do \emph{not imply} that  $(X_0,X_1)$ is not sequentially related by $\FamilyVanishing\bigcup \FamilyVanishingMir$.
The manifolds $X_0$ of Theorem~\ref{thm:DeltaZeroNotUniv} can be taken to be minimal complex surfaces of general type of signature one.

By showing that $(W_n, f_{W_n}) \in \FamilyVanishingMir$ for all $n$, we obtain the following corollary.
\begin{corollary}\label{thm1}	
	For any $n\geq 1$, the corks $(W_n, f_{W_n})$ and $(-W_n,f_{W_n})$ are not universal.
\end{corollary}	
We will also see  that  Corollary~\ref{thm1} depends only on the cork boundary, indeed any cork with boundary $\pm \partial W_n$ is not universal (see Corollary~\ref{Cor:invNotRel}).

Theorem~\ref{thm:DeltaZeroNotUniv} shows a stark contrast between the $(W_n, f_{W_n})$-twist and the  generalized logarithmic transform, i.e. the operation consisting in removing an embedded $\mathbb{T}^2\times \mathbb D^2$ and gluing it back via a diffeomorphism of its boundary. Indeed,  any pair of  simply-connected, closed $4$-manifolds is related by \emph{a sequence} of generalized logarithmic transforms \cite[Cor. 11]{BaykurSunukjianRoundHandles}.
%We will also see  that  the involution on $ \partial W_n$ is not really relevant for Corollary~\ref{thm1}, indeed any cork of the form $(\pm W_n, f)$ is not universal (see Corollary~\ref{Cor:invNotRel}).

In the setting of manifolds with boundary it is still true that any exotic pair $(X_0,X_1)$ is related by a cork, provided that $\partial X_0$ is connected and $\pi_1(X_0) = 1$ \cite{boyer_1986}  (see also Theorem~\ref{Thm:corkDec} below).
Therefore it makes  sense to investigate the universal cork conjecture in this setting.

\begin{definition} We say that a  cork $(C,f)$  is $\partial$-universal if every exotic pair $(X_0,X_1)$ of simply-connected  $4$-manifolds with connected boundary, is related by $(C,f)$.
\end{definition}

Then we can formulate the following $\partial$-universal cork conjecture: \emph{there is no $\partial$-universal cork}.
In this case we are able to prove the conjecture in its entirety:
\begin{theorem}\label{thm2}
	There exists no $\partial$-universal cork.
\end{theorem}

The proofs of Theorem~\ref{thm:DeltaZeroNotUniv} and Theorem~\ref{thm2} are rather different.
In the former case  we infer some results about how cork twists by $W_n$ affect the mod-$2$ Seiberg-Witten invariants  and then we look for exotic pairs with a large enough variation of these invariants.
On the other hand, the proof of Theorem~\ref{thm2} is based  on finding a sequence of exotic pairs of diverging complexity (Definition~\ref{def:complexityPair}), this sequence is obtained by applying the Akbulut-Ruberman construction \cite{AkbulutRubermanAbsolutelyExotic} to the corks obtained by applying the cork decomposition theorem \cite{CHFS, Matveyev} to the inertial cobordisms of diverging complexity of \cite{MorganSzabo99}.

\paragraph{Related results.} \renewcommand{\cX}{\mathcal{X}}
We want to clarify the difference between our result and other results about smooth structures not related by corks.  
Let $G$ be a group. A $G$-cork \cite[Def. 1.6]{TangeNonExtInfCork} is a contractible $4$-manifold $C$ with a $G$-action   $\alpha: G\to \mathrm{Diff}^+(\partial C)$ such that for any  $g\in G$, the pair $(C,\alpha(g))$ is a cork (here we do not require $\alpha(g)$ to be an involution). Clearly this definition specializes to our previous one when $G=\Z/2\Z$.

Tange \cite[Thm. 1.4]{TangeNonExtInfCork} proved that there exists an infinite family of  closed, simply-connected $4$-manifolds $\cX = \{X_n\}_{n\in N}$ all homeomorphic but pairwise not diffeomorphic, such that for any group $G$ and any $G$-cork $(C,\alpha)$,
there does not exist an embedding $\emb: C\hookrightarrow X_0$ such that $\cX$ is a subset (up to diffeomorphism) of the set of manifolds
\begin{equation*}
	\big\{ X_0\setminus \interior{\emb(C)} \bigcup_{\emb\circ \alpha(g)} C \  \big | \ g \in G \big\},
\end{equation*}
i.e. $\cX$ is \emph{not} generated by the $G$-cork $(C,\alpha)$.
It is necessary  for the family $\cX$ to be infinite, indeed Melvin and Schwartz  \cite{MelvinSchwartz}  proved that if  a family $\cX$ as above consists of $p\in \N$ elements, then it can be generated by embedding a $\Z/p\Z$-cork. \cite{MelvinSchwartz}  can be thought as a generalization
of the cork decomposition theorem \cite{CHFS,Matveyev} for $\Z/p\Z$-corks.

Tange's result has then been  generalized by Yasui \cite[Thm. 1.3]{YasuiNonExtSurgeries} 	who showed that for any $n\in \N$ there is a closed simply-connected $4$-manifold $Q_n$ such that for any (not necessarily connected) codimension-$0$ submanifold $W\subset Q_n$
with $b_1(\partial W)< n$ exists an exotic pair $(Q_n, \tilde Q_n)$ such that $\tilde Q_n$ is not diffeomorphic to any of the manifolds obtained  by cutting out $W$ from $Q_n$ and gluing it back with a diffeomorphism of $\partial W$ (the gluing map can be anything). 
Both the authors  obtained analogous results in the setting of manifolds with boundary  \cite{TangeNonExtInfCork, YasuiNonExtSurgeries}.

Tange's and Yasui's  results do not imply ours  because they fix the embedding, thus, in principle, it is still possible 
to relate every exotic pair $(Q_n, \tilde Q_n)$ by embedding the same $\Z/2\Z$-cork $(C,f)$ in different ways.

Regarding the $\partial$-universal cork conjecture, it follows from \cite{KangOneStabContractible} that 
any cork arising from an $h$-cobordism with an handle decomposition with only one $2$-handle and one $3$-handle is not $\partial$-universal (in the language of \cite{Ladu}, these are the corks admitting a supporting protocork of sphere-number one). In particular,  \cite{KangOneStabContractible}  implies that the corks $(W_n, f_{W_n})$ are not $\partial$-universal.

\paragraph{Structure of the paper.	} In Section~\ref{sec:Conventions} we set up our notation and conventions and we review the definition of cork and of cork relating and exotic pair. In Section~\ref{sec:proofThm1} we prove  Theorem~\ref{thm:DeltaZeroNotUniv} and Corollary~\ref{thm1}, in  Section~\ref{sec:proofThm2}
we prove Theorem~\ref{thm2}.

\begin{acknowledgements}  The author would like to thank Arunima Ray, Daniel Ruberman and  Steven Sivek for their insightful comments and for reviewing a  draft of this paper, Dieter Kotschick for useful e-mail exchange and an anonymous referee for useful comments on a previous version of this paper.
The author is grateful to the Max Planck Institute for Mathematics in Bonn for hospitality and financial support.
\end{acknowledgements}

\section{Basic definitions, conventions.}\label{sec:Conventions}
\subsection{Conventions and basic definitions.}
	In this paper all the manifolds will be compact, \emph{oriented} and smooth.
	Homeomorphisms, diffeomorphisms and codimension-$0$ embeddings are \emph{always ment to preserve orientation, 
	even if we do not write it explicitely}, but sometimes we will remark it for extra clearness.	
	We will adopt the notation $\interior{X} = X \setminus \partial X$ to  denote the interior of a manifold.
	 The symbols $\sim$, $\approx$ and  $\cong$  will denote respectively  the relations of homotopy equivalence,  (orientation preserving) homeomorphism and (orientation preserving) diffeomorphism.
	 Given two $n$-manifolds   $N^n,M^n$ and a diffeomorphism $f$ from some components of $\partial M$ 
	to  $\partial N$, we denote by  $N\bigcup_f M$ the manifold obtained by gluing  $M$ to $N$ using $f$.

	We recall the definition of cork \cite{AkbulutBook}, some authors require corks to be Stein, but we will not need this extra structure.
	\begin{definition}\label{def:cork} An (abstract) \emph{ cork} is a pair $(W,\tau)$ where $W$ is a compact, \emph{contractible}, oriented, smooth $4$-manifold  with $\partial W\neq \emptyset$ and $\tau:\partial W \to \partial W$ 
is an (orientation preserving) involution such that  $\tau$  does \emph{not} extend to an (orientation preserving) diffeomorphism of $W$.
	\end{definition}
	Notice that by \cite{Freedman, boyer_1986} the involution of a cork will always extend to an homeomorphism of the full manifold.
	
	\begin{definition}\label{Def:related}
		A pair of $4$-manifolds $(X_0,X_1)$ (not necessarily closed)  is \emph{related by a cork} $(C,f)$ if exists an embedding $\emb: C\to X_0$ and a diffeomorphism
		\begin{equation*}
			F:  X_0\setminus \interior{\emb(C)} \bigcup_{\emb\circ f} C \to  X_1.	
		\end{equation*} 
	\end{definition}
	We will refer to  the operation which associates to $X_0$  the manifold $	X_0\setminus \interior{\emb(C)} \bigcup_{\emb\circ f} C $ as   \emph{cork twist} by $(C,f)$ or just as \emph{$(C,f)$-twist}.	Notice that in the above definition we are not requiring anything about the behaviour of $F$ on the boundary of $X_0$.
	We remark that, following our conventions above, the embedding $\emb$ and the diffeomorphism $F$ need to be orientation preserving.
	
	\begin{remark}
		The choice of the embedding is crucial, for example Akbulut and Yasui \cite{AkbulutYasuiKnottingCorks}  produced infinitely many exotic copies of a closed, simply connected $4$-manifold $X$ which are not pairwise diffeomorphic by embedding the same cork in  $X$ in different ways.
		In particular, for  a cork $(C,f)$  to not relate an exotic pair $(X_0, X_1)$, is a property which encompasses
		all the embeddings of  $C$ in $X_0$.
	\end{remark}

	\begin{definition}\label{Def:sequentiallyRelated}  Let $\mathcal{S}$ be a set of corks.
	A pair of $4$-manifolds $(X_0,X_1)$ (not necessarily closed)  is \emph{sequentially related  by $\mathcal S$},
		if there exists a sequence of $4$-manifolds $M_0, M_1,\dots, M_m$, $m\geq 0$, with $M_0\cong X_0$ and $M_m \cong X_1$,
		such that, for any $i=0,\dots, m-1$,  $(M_i, M_{i+1})$ is related by some cork in $\mathcal{S}$.
	\end{definition} 
 
 It is well known that any  exotic pair $(X_0, X_1)$ of simply-connected, \emph{closed}
$4$-manifolds is related by some cork  \cite{CHFS, Matveyev}. The analogous statement for manifolds with boundary follows from \cite{boyer_1986} and is less known, hence we include it here together with a proof sketch.

	\begin{theorem}[Boyer \cite{boyer_1986}, Curtis-Freedman-Hsiang-Stong \cite{CHFS}, Matveyev \cite{Matveyev}]\label{Thm:corkDec} Let $V_1, V_2$ be  simply-connected $4$-manifolds with connected boundary and let $F:V_1\to V_2$ be an homeomorphism extending the diffeomorphism $f:\partial V_1\to \partial V_2$.
	Then $V_1$ and $V_2$ are $h$-cobordant relative to $f$ and, if  $V_1\not \cong V_2$, there is a cork relating $(V_1, V_2)$.
\end{theorem}
\begin{proof}[Proof sketch] The $h$-cobordism $V_1\to V_2$ of the thesis is constructed in \cite[Prop. 4.2]{boyer_1986}.  We review its construction.

			Glue $V_1$ to $V_2$ using $f$ obtaining  $V:= V_1\bigcup_f -V_2$.	
			Since $f$ extends to an homeomorphism, $V$ is homeomorphic to $V\approx V_1\bigcup_{id} -V_1$, the double of $V_1$, and the latter is homeomorphic to  $S$, with $S =\#_{i=1}^{b_2(V_1)} \SS^2\times \SS^2$ if $V$ is spin and   $S=\#_{i=1}^{b_2(V_1)}  \CP^2\#\overline{\CP}^2 $ otherwise \cite{Freedman}.
			
	The homeomorphism $F$ defines a ``doubling homomorphism" $H_2(V_1, \partial V_1)\to H_2(V)$, by associating to a relative cycle $\Sigma \in Z_2(V_1,\partial V_1)$ the absolute cycle  $\Sigma-F(\Sigma) \in Z_2(V)$  obtained by gluing $\Sigma$ to $F(\Sigma)\in Z_2(V_2, \partial V_2)$.
	Let $J\subset H_2(V)$ be the image of the doubling homomorphism. This is a maximal isotropic subgroup of $H_2(V)$  and $\partial(J) = H_1(\partial V_1) $ 	where $\partial: H_2(V) \to H_1(\partial V_1)$ is the connecting morphism of the Mayer-Vietoris exact sequence for  $V = V_1\cup_f V_2$. 
	
	By \cite{Wall} there exists an $h$-cobordism $Z:V\to S$.  The manifold $S $  bounds $E$,  a (possibly  twisted) $\D^3$ bundle over $\SS^2$. Glue  $E$ to $Z$ along $S$ obtaining $Z' = Z \bigcup_S E$, a null-cobordism of $V$, with $Z'\sim \bigvee^{b_2(V)}_{i=1} \SS^2$.
	Since $J$ is maximal isotropic, it is possible to  reglue $E$ to $Z$,  using a suitable diffeomorphism  $\varphi \in \mathrm{Diff}^+(S)$, obtaining $Z'' = Z \bigcup_\varphi E$, 	 in such a way that $\ker (H_2(V)\to H_2(Z'')) = J$ and $Z''\sim \bigvee^{b_2(V)}_{i=1} \SS^2$ \cite[pg 145]{Wall}.
	
	To show that $Z'':V_1\to V_2$ is an $h$-cobordism  relative to $f$, now it is enough to show that $H_2(V_i)\to H_2(Z'')$, $i=1,2$, is surjective (the injectivity will follow from the equality of the ranks) 	which is proved by the commutative diagram shown in \cite[pg. 348]{boyer_1986}.

	Now, since $\pi_1 (V_i) = 1$, $i=1,2$, we can repeat the proof of the cork decomposition theorem \cite{CHFS,Matveyev} starting with the $h$-cobordism $Z''$,   this will yield 	a cork  inducing $Z''$ hence relating $(V_1, V_2)$.
	The assumption of $V_1\not \cong V_2$, is used to ensure that the involution of the cork does not extend to a diffeomorphism.
\end{proof}

\section{Proof of Theorem~\ref{thm:DeltaZeroNotUniv}.}\label{sec:proofThm1}

\subsection{The difference element of corks.}\label{sec:DifferenceElement}
We will make use of monopole Floer homology, as developed in \cite{KM}. In the following $\F$ will denotes the field of two elements.
Recall that for any cobordism $X:Y_0\to Y_1$ between closed, oriented $3$-manifolds, there is an associated cobordism map $\HMfrom(X;\F): \HMfrom(Y_0;\F)\to \HMfrom(Y_1;\F)$ which is an $\F[U]$-homomorphism~\cite[Sec. 3.4]{KM}.

Let $(C,f)$ be a cork. Remove a ball from $C$ and interpret the resulting manifold $ C\setminus \ball$ as  a cobordism $\SS^3\to \partial C$.
The mod-$2$ relative Seiberg-Witten invariant of $C$ is an element of the Floer homology of $\partial C$ with coefficients in $\F$ and is  defined as 
\begin{equation*}
	\HMfrom(C\setminus \ball; \F)(\hat 1)  \in \HMfrom(\partial C;\F),
\end{equation*}
where $\hat 1\in \HMfrom(\SS^3;\F)$ is the canonical generator.

In \cite{Ladu} we studied the difference element of protocorks and corks using integer coefficients, here we will only need its mod-$2$ reduction.\begin{definition} Given a cork $(C,f)$ let $x\in \HMfrom(\partial C;\F) $ be the mod-$2$ relative Seiberg-Witten invariant of $C$.
Then we define the \emph{difference element} of $(C,f)$ to be 
	\begin{equation*}
		\Delta = x - f_* x \in \HMfrom(\partial C; \F),
	\end{equation*}
	where $f_*:\HMfrom(\partial C; \F)\to\HMfrom(\partial C; \F)$ is the involution  induced by $f$.
\end{definition}

Recall that if $Y$ is an integer  homology sphere, then $\HMfrom(Y;\F)$ is an absolutely $\Q$-graded $\F[U]$-module where multiplication by $U$
is a degree $-2$ homomorphism. Furthermore there is a splitting $\HMfrom(Y;\F)\simeq \hat \cT_{-2h-1} \oplus \HMred(Y;\F)$
where $h \in \Z$, is the Fr{\o}yshov invariant of $Y$ \cite[Sec. 39]{KM},  
 $\hat \cT_{-2h-1}$ is isomorphic to the graded $\F[U]$-module $\F[[U]]$ with $1\in \F[[U]]$ being of degree $-2h-1$, and 
$ \HMred(Y;\F)$ is the \emph{reduced homology} which is the  $U$-torsion submodule of~$\HMfrom(Y;\F)$. 

The next proposition is essentially  \cite[Cor. 1.3]{Ladu}. The fact that  $-W_1$ cannot change the Seiberg-Witten invariants (even with $\Z$-coefficients) was first noticed in \cite[Thm. 8.6]{LinRubermanSaveliev2018} by studing the action of the involution in Floer homology. 

\begin{proposition}\label{prop:DeltaVanNoVar} Let  $(C,f)$ be a cork and let $\Delta$ be its difference element, then:
\begin{enumerate}
	\item \label{prop:DeltaVanNoVara} $\Delta$ belongs to $\HMred_{-1}(\partial C; \F)$.
	\item  \label{prop:DeltaVanNoVarb} if $\Delta = 0$, and $(X_0, X_1)$ is a pair of homeomorphic, closed, oriented, simply-connected $4$-manifolds  related by the cork $(C,f)$ with $b^+(X_0)\geq 2$, then there exists an isometry $\beta: H^2(X_0)\to H^2(X_1)$  such that $SW(X_0,  K) \equiv SW(X_1, \beta(K)) \mod 2$ for any characteristic element $K\in H_2(X_0)$.
\end{enumerate} 
\end{proposition}
Informally, the second item tells us that  if the difference element vanishes then cork twists by $(C,f)$ cannot change the Seiberg-Witten invariants.
\begin{proof}
 Since we are using $\F$-coefficients, Corollary 1.3 of \cite{Ladu}   implies that $\Delta \in \HMred_*(Y;\F)$. The fact that the 
 the degree  of $\Delta$ is $-1$  follows from the fact that
 $f_*$ preserves the grading and the relative Seiberg-Witten invariant of a cork lives in degree $-1$ because
$\hat 1 \in \widehat{HM}_{-1}(\SS^3)$ and the cobordism map $\HMfrom(C\setminus \ball)	$ has degree equal to zero
because $C$ is contractible \cite[Eq. (28.3)]{KM}.

The second part is an application of  \cite[Prop. 3.6.1]{KM} (see also the proof of Corollary 1.2 in \cite{Ladu}).
\end{proof}

\subsection{Proof of Theorem~\ref{thm:DeltaZeroNotUniv} and Corollary~\ref{thm1}}

We begin with the following observation.   
\begin{lemma} \label{lemma:CMinusCrelates} If $(C,f)$ relates the  pair $(X_0, X_1) $ then $(-C,f)$ relates $(-X_0, -X_1)$.
\end{lemma}
\begin{proof}
By definition, there exists an (orientation preserving) embedding $\emb:C\to X_0$
and a diffeomorphism $F: X_0\setminus \interior{\emb(C)}\bigcup_{\emb\circ f} C \to X_1.$

By reversing the orientation on both sides, $F$ gives also a diffeomorphism 
\begin{equation*}
	F: (-X_0)\setminus \interior{\emb(-C)}\bigcup_{\emb\circ f} -C \to -X_1,
\end{equation*}
which implies that $(-X_0,-X_1)$ is related by the cork $(-C,f)$.	
\end{proof}

We are now ready to prove Theorem~\ref{thm:DeltaZeroNotUniv} which we restate below for the reader's convenience.
\begin{definition}	We define the following sets of corks
	\begin{spliteq*}
		&\FamilyVanishing  = \{(C,f) \ \text{cork} \ | \ \Delta_{(C,f)}= 0 \},  \\
		& \FamilyVanishingMir = \{(C,f)\  \text{cork} \ | \ (-C,f) \in \FamilyVanishing\},  \\
	\end{spliteq*}
where $\Delta_{(C,f)}$ is the difference element of $(C,f)$.
\end{definition}

\begingroup
\def\thetheorem{\ref{thm:DeltaZeroNotUniv}}
\begin{theorem}
There exist infinitely many exotic pairs $(X_0,X_1)$ of simply-connected, closed $4$-manifolds,
such that  \begin{enumerate}
					\item \label{itema} $(X_0,X_1)$ and $(-X_0, -X_1)$ are not  sequentially related  by $\FamilyVanishing$, 
					\item \label{itemb} $(X_0,X_1)$ and $(-X_0, -X_1)$ are not  sequentially related  by $\FamilyVanishingMir$.
				\end{enumerate}
\end{theorem}
\addtocounter{theorem}{-1}
\endgroup

\begin{proof}[Proof of Theorem~\ref{thm:DeltaZeroNotUniv}]
Let $X_0$ be a (closed) simply-connected, non-spin, minimal surface of general type with $\sigma(X_0)=1$, $b^+(X_0)\geq 50$. Infinitely many such manifolds are given by Chen surfaces \cite[Thm. 3.7]{KotschickOrReversalGeography}.
Let $\tilde X$ be a  (closed)  simply-connected, symplectic $4$-manifold with $\sigma(\tilde X) = 0$, $b^+(\tilde X) = b^+(X_0)-1$.
There are infinitely many such manifolds for any $X_0$ as above since $b^+(\tilde X)\geq 49$ \cite{AkhmedovHughesParkSymplecticPosSignature}.
Now we set  $X_1 := (-\tilde X)\#\CP^2$.

We will show that the pairs $(X_0, X_1)$ and $(-X_0,-X_1)$ are exotic and are not related by neither $(C,f)$ nor $(-C,f)$.
To see that $X_0$ is homeomorphic to $X_1$ notice that 
both $X_0, X_1$ are non-spin, have signature one and  $b^+(X_0) = b^+(X_1)$, consequently $X_0\approx X_1\approx b^+(X_0) \CP^2\# (b^+(X_0)-1)\overline{\CP}^2$ \cite{Freedman}.  

\newcommand{\SWModTwo}{SW_{\frac{\Z}{2\Z}}}
\newcommand{\basicClasses}{\cB_{\frac{\Z}{2\Z}}}
For a simply-connected $4$-manifold $M$, we denote by  $\basicClasses(M^4)\subset H^2(M)$ the set of basic classes  for the  mod-$2$ Seiberg-Witten invariant \cite[Def. 2.4.5]{GompfStipsicz}.

Since $X_0$ is a complex surface of general type, with $b^+(X_0)\geq 2$,  $\basicClasses(X_0)\neq \emptyset$ \cite{FriedmanMorganAlgebraicSW}. Moreover, since $\sigma(X_0) \neq 0$,  $\pi_1(X_0) = 1$,  \cite[Thm. 1]{KotschickGeometrization} implies that $\basicClasses(-X_0) = \emptyset$.
Being the blow-up of a symplectic manifold,   $-X_1$ is symplectic, hence $\basicClasses(-X_1)\neq \emptyset$ by \cite{TaubesSymplectic}. On the other hand, $X_1$ contains an embedded sphere of self-intersection one and $b^+(X_1)\geq 2$, thus $\basicClasses(X_1) \equiv 0$  \cite{fintushel_stern_blowup}. 
This shows that $X_0 \not \cong X_1$.

To prove item-\ref{itema}, now notice that if  $M_0, M_1,\dots,  M_m$, $m\in \N$ is a sequence of simply-connected, closed, oriented $4$-manifolds such that for each $i<m$
 $(M_i, M_{i+1})$ is related by  a cork in $\FamilyVanishing$, then Proposition~\ref{prop:DeltaVanNoVar}~\ref{prop:DeltaVanNoVarb} implies
 that $|\basicClasses(\pm M_0)| =  |\basicClasses(\pm M_m)|$, i.e. the cardinality of the set of basic classes is constant along the sequence.
 On the other hand $|\basicClasses(\pm X_0)| \neq  |\basicClasses(\pm X_1)|$.
 
Item-\ref{itemb} follows from Lemma~\ref{lemma:CMinusCrelates} and item-\ref{itema}.
\end{proof}

\begin{proof}[Proof of Corollary~\ref{thm1}] It is enough to show that $-W_n\in \FamilyVanishing$ and invoke  Theorem~\ref{thm:DeltaZeroNotUniv}.
By Proposition~\ref{prop:DeltaVanNoVar}~\ref{prop:DeltaVanNoVara} to see that the difference element  vanishes it is sufficient to see that~${\HMred_{-1}(-\partial W_n; \F) = 0}$.

In the case $n=1$,  this follows from \cite{AkbulutDurusoy} and  the cases $n=2,3$ follow from \cite{AkbulutKarakurtHFMazur}.
More generally it follows from \cite[Thm. 1.0.3]{HalesThesis} that for any $n$, $\HMred(-\partial W_n; \F) $ takes the form 
\begin{equation*}
	\HMred(-\partial W_n; \F) \approx \bigoplus_{s,i \in \Z} M(i,s)_{((s^2-s)- (i^2-i)- \min(0, 2(i-s)))},
\end{equation*}
where $M(i,s)_{r}$ is some torsion $\F[U]$-module supported in absolute grading $r$.
Since for any $s,i\in \Z$, $((s^2-s)- (i^2-i)- \min(0, 2(i-s)))$ is even, we see that the reduced homology of $-\partial W_n$ is supported in even degrees which proves the assertion.
\end{proof}

Notice that if  $\HMred_{-1}(\partial C; \F) = 0$ then $\Delta_{(C,f)} = 0$ for any involution $f$, thus  being not universal does not depend on the specific involution of the cork. We record this fact in a corollary for future use.

\begin{corollary} \label{Cor:invNotRel}Let $(C,f)$ be a cork with $\HMred_{-1}(\partial C; \F) = 0$, then any cork of the form $(\pm C, g)$, $g\in \mathrm{Diff}^+(\partial C)$ belongs to $\FamilyVanishingMir\bigcup \FamilyVanishingMir$ hence is not universal (recall that in the definition of cork we require $g$ to be an involution).
\end{corollary}

%%%%%%%%%%%%%%%%%%%%%%%%%%%%%%%%%%%%%%%%%%%%%%%%%%%%%%%%%%%%%%%%
%%%%%%%%%%%%%%%%%%%%%%%%%%%%%%%%%%%%%%%%%%%%%%%%%%%%%%%%%%%%%%%%
\section{Proof of Theorem~\ref{thm2}}\label{sec:proofThm2}

\subsection{Notions of complexity.}\label{subsec:ComplexityDef}
\newcommand{\Complexity}{\mathfrak{C}}

We begin by recalling the definition of complexity of an $h$-cobordism introduced by Morgan and Szab\'o in \cite{MorganSzabo99}.
Let $Z:X_0\to X_1$ be an $h$-cobordism between simply-connected $4$-manifolds.
 If $\partial X_0\neq \emptyset$ we further suppose that
$Z$ restricts to a \emph{cylindrical} cobordism $\partial X_0 \to \partial X_1$  denoted by $Z|_{\partial X_0}$. In this case, a preferred diffeomorphism $\Psi_\partial^Z:Z|_{\partial X_0}\overset{\cong}{\to} \partial X_0 \times [0,1]$ is part of the data defining $Z$.
We recall that a cylindrical cobordism is a cobordism where the underlying manifold is diffeomorphic to a cylinder, however  the maps identifying the boundary components need not be trivial.

Denote by $\cH(Z, X_0, X_1)$ the set of handle decompositions $\Delta$ of $Z$ such that:
			\begin{enumerate}
			\item $\Delta$ consists of $n\geq 0$   $2$-handles $\{h^2_j\}_{j=1}^n$ and $n$ $3$-handles $\{h^3_i\}_{j=1}^n$,
			The  $2$-handles are attached simultaneously to the interior of $X_0\times \{1\}\subset X_0\times [0,1]$ obtaining  $Z'$
			and the $3$-handles are then attached simultaneously to the interior of $\partial_+ Z'$.
		    
			\item Let the 4-manifold $Z_{1/2}\subset Z$ be  the middle 
			level of the cobordism obtained after attaching all the $2$-handles. 
			 Denote by $B_j\subset Z_{1/2}$ be the belt sphere of $h^2_j$ and by $A_i\subset Z_{1/2}$ be the attaching sphere of $h^3_i$.  Then we require that their algebraic intersection number in $Z_{1/2}$
			 is equal to the Kronecker delta:	$B_j\cdot A_i = \delta_{i j}$ and the normal bundle of $B_j$, $A_i$ is trivial for any $i,j$.

			 \item If $\partial X_0\neq \emptyset$ then the handle decomposition is induced by a Morse function $(Z,\partial_- Z, \partial_+ Z)\to ([0,1], 0,1)$ which on $Z|_{\partial X_0}$  is the pull-back via $\Psi_\partial^Z$
			 of the projection $\partial X_0\times [0,1]\to [0,1]$.
			\end{enumerate}
			
Given such  an  handle decomposition $\Delta $  we define its complexity as
\begin{equation*}
	\Complexity(\Delta) := \left(\sum_{i,j=1}^n|A_i\cap B_j|\right) - n \in \N,
\end{equation*}
i.e. the number of geometric intersections between the belt spheres of the $2$-handles and the attaching spheres of the $3$-handles minus 
the number of  $2$-handles. 

The \emph{complexity} of the $h$-cobordism $Z$ is then defined as
\begin{equation*}
	\Complexity(Z) := \min_{\Delta \in \cH(Z,X_0, X_1)} \Complexity(\Delta).
\end{equation*}

The proof of the h-cobordism  theorem (\cite{SmaleHCob}, see also \cite{MilnorLecturesHCob}) ensures that $\cH(Z,X_0,X_1)$ is not empty.

\begin{definition} Let $(C,f)$ be a contractible manifold with an involution of its boundary, we define the complexity of $(C,f)$ as
	\begin{equation*}
		\Complexity(C,f) := \min_{\substack{Z:C\to C \\ \text{$h$-cob.  rel $f$}}} \Complexity(Z),
	\end{equation*}
		i.e. the minimum of the complexities of   $h$-cobordisms $Z:C\to C$ relative to $f$. 
\end{definition}
By an $h$-cobordism $Z:C\to C$ relative to $f$ we mean that the map $\Psi_\partial^Z $ induces a map $\partial C\to \partial C$ which is equal to $f$.  
Notice that such $h$-cobordisms always exist, indeed  $C$ being contractible implies that  the twisted double $C\cup_f - C$ bounds a contractible $5$-manifold (c.f. the proof of \cite[Prop. 2.12]{Ladu}). It can be shown, but we will not need it, that this is the \emph{unique} $h$-cobordism relative to $f$  up to isomorphism \cite{Kreck2001}.

\begin{definition} \label{def:complexityPair}Let $(X_0, X_1)$ be a pair of \emph{homeomorphic}, simply-connected $4$-manifolds, possibly with connected boundary.
We define the complexity of the pair $(X_0, X_1)$ as
	\begin{equation*}
		\Complexity(X_0, X_1) := \min_{\substack{Z:X_0\to X_1 \\ \text{$h$-cobordism}}} \Complexity(Z),
	\end{equation*}
	i.e. the minimum of the complexities of   $h$-cobordisms $Z:X_0\to X_1$.
\end{definition}
The existence of an $h$-cobordism is guaranteed by  Theorem~\ref{Thm:corkDec}.

%In the closed case, the existence of an $h$-cobordism is guaranteed by \cite{Wall}.
%In the case $\partial X_0\neq \emptyset$, the existence of an $h$-cobordisms $Z:X_0\to X_1$ follows from~\cite{boyer_1986}.

%We have the following equivalent definition.
%\begin{lemma}\label{lem:cplPairCorks} 
%	\begin{equation*}
%		\Complexity(X_0, X_1) \geq  \frac 1 2  \min_{\substack{(C,f) \text{ cork} \\ \text{relating $(X_0,X_1)$}}} \Complexity(C,f),
%	\end{equation*}
%
%\end{lemma}
%\begin{proof}
%The proof of the cork decomposition theorem \cite{CHFS, Matveyev} applies to simply-connected $h$-cobordism $Z:X_0\to X_1$, regardless of $X_0$ having boundary, and shows that given an handle decomposition $\Delta \in \cH(Z,X_0, X_1)$,
% there exists a cork $(C,f)$, with $\Complexity(C,f)\leq 2\cdot \Complexity(\Delta)$, such that  $Z$ is a trivial cobordism outside of a  subcobordism isomorphic to an $h$-cobordism $C\to C$ relative to $f$. 
%\end{proof}

\subsection{Upper-bound on the complexity.}
In this subsection  $Z:X_0\to X_1$ will be an $h$-cobordism  of the type considered in Subsection~\ref{subsec:ComplexityDef}, in particular $\pi_1(X_0)=1$ and $Z|_{\partial X_0}$ is cylindrical.

Pick an handle decomposition $\Delta \in \cH(Z,X_0, X_1)$ and consider $P_{1/2}$,  a tubular neighbourhood of $\bigcup_{j} (A_j\cup B_j)$ in $Z_{1/2}$.
We define  the cobordism $P(\Delta):P_0(\Delta)\to P_1(\Delta)$ to be the cobordism obtained from $P_{1/2}\times [0,1]:P_{1/2}\to P_{1/2}$ by gluing, for each $j$, the handle $h^2_j$ of $\Delta$ along its \emph{belt} sphere  $B_j\hookrightarrow P_{1/2}\times\{0\}$ and the handle $h^3_j$ of $\Delta$ along its attaching sphere $A_j\hookrightarrow P_{1/2}\times\{1\}$.

 \begin{remark} \label{rem:PcorkTwist} $P(\Delta) $ is an $h$-cobordism of complexity $\Complexity(P(\Delta)) \leq \Complexity(\Delta)$ and, by construction, $P(\Delta)$   embeds as a subcobordism of $Z:X_0\to X_1$ in such a way that 
 \begin{equation*}
 	 X_1 \cong X_0 \setminus \interior{(P_0(\Delta))}\bigcup P_1(\Delta).
 \end{equation*}
  \end{remark}

Notice that there is  an identification $\partial P_{1/2} = \partial P_0(\Delta) = \partial P_1(\Delta)$ 
 since  the manifolds $P_{i}(\Delta)$, $i=0,1$,  are obtained by surgerying spheres in the interior of~$P_{1/2}$.

\begin{definition} We call  the $h$-cobordism $P(\Delta):P_0(\Delta)\to P_1(\Delta)$  constructed from $\Delta$ as above the  \emph{protocork cobordism} associated with $\Delta$ and we say that  the triple $(P_0(\Delta), P_1(\Delta), id)$, with $id: \partial P_0(\Delta)\to \partial P_1(\Delta) $  being the above identification, is the \emph{abstract protocork} associated with $\Delta$.
 \end{definition}
Clearly $(P_0(\Delta), P_1(\Delta), id)$ is an  abstract (possibly non-symmetric) protocork in the sense  of \cite[Def. 2.10]{Ladu}.
Next we generalize the definition of supporting protocork \cite{Ladu} to include non-symmetric protocorks.
\begin{definition} \label{def:supporting} Let $(C,f)$ be a cork.
We say that $(P_0(\Delta), P_1(\Delta), id)$    \emph{supports}  $(C,f)$ if there
exists an  embedding $\emb: P_0(\Delta)\to \interior{C}$ and a diffeomorphism  
\begin{equation*}
		F: C \to C\setminus \interior{\emb(P_0(\Delta))}\bigcup_{\emb} P_1(\Delta)
\end{equation*}
restricting to $f:\partial C \to \partial C$ on the boundary.
\end{definition}

\begin{lemma}\label{lem:cplCorkSupp}  Let $(C,f)$ be a cork. 
If the protocork $(P_0(\Delta), P_1(\Delta), id)$ supports $(C,f)$ then $\Complexity(C,f)\leq~\Complexity(\Delta)$.
\end{lemma}

\begin{proof} Let $\emb$ and $F$ as in Definition~\ref{def:supporting}. 
Glue $P(\Delta)$ to the trivial cobordism $C\times I$ using the map  $P_0(\Delta)\overset{\emb}{\to} C\times\{1\}$, obtaining 
a cobordism 
\begin{equation*}
	V: C\to C\setminus \interior{\emb(P_0(\Delta))}\bigcup_{\emb} P_1(\Delta). 
\end{equation*}
 $V$ is an  $h$-cobordism with an handle decomposition of complexity $\Complexity(\Delta)$ induced by the handle decomposition defining  $P(\Delta)$.

Now construct a cobordism $V':C\to C$ which is equal to the cobordism $V$ except for the identification map of the outcoming boundary
which is now given by $F: C\times\{0\}\to \partial_+ V$.
Consequently, $V'$ is an $h$-cobordism relative to $f:\partial C\to \partial C$ because by assumption $F$ restricts to $f$ on the boundary.

Now $\Complexity(C,f)\leq \Complexity(V') = \Complexity(V) \leq \Complexity(\Delta)$.
Where the equality is due to the fact that changing the identification map of the outcoming boundary does not affect the complexity because it does not change the set of Morse functions for the cobordism.
\end{proof}

We will also make use of the following.
\begin{lemma} \label{lem:CorkRelExoticPair} Let $(C,f)$ be a cork.  If $(X_0, X_1)$ is related by $(C,f)$ then $\Complexity(X_0, X_1) \leq \Complexity(C,f)$.
\end{lemma}
\begin{proof} Let $\emb: C\to X_0$ denote the embedding of Definition~\ref{Def:related} and let $M = X_0 \setminus (\interior{\emb(C)})$.
	Let $Z:C\to C$ be an $h$-cobordism relative to $f$ realizing $ \Complexity(C,f)$.
	By gluing $M\times I$ to $Z$ via the diffeomorphism
	\begin{equation*}
	 Z|_{\partial C}\overset{\Psi^Z_\partial}{\to}  \partial C \times I \overset{\emb \times id_I}{\to} \partial M \times I,
	\end{equation*}	 
	we obtain an $h$-cobordism $V$ from $X_0$ to a manifold diffeomorphic to $X_1$ which has complexity lower or equal than $\Complexity(Z)$ because a Morse function for $Z$ extends to a Morse function for $V$ without critical points over $M\times I$.
\end{proof}

\subsection{The Akbulut-Ruberman construction.}
Given a cobordism $X:Y_0\to Y_1$ we will denote by $X^\dagger:Y_1\to Y_0$
the cobordism obtained by changing the orientation on $X$ and swapping the roles of incoming and outcoming end.

We have the following result due to  Akbulut and Ruberman.
\begin{theorem}[Akbulut-Ruberman \cite{AkbulutRubermanAbsolutelyExotic}]\label{AkbulutRuberman} Let $(C,f)$ be a cork, then there exist a $3$-manifold $N$ and a cobordism $Q: \partial C\to N$, canonically associated with $(C,f)$, such that 
\begin{enumerate}
\item  \label{AkbulutRubermana} the composite cobordism $Q^\dagger\circ Q$ is isomorphic to the product cobordism $\partial C \times I$,
\item \label{AkbulutRubermanb}the group of diffeomorphisms of $N$ modulo isotopy is isomorphic to $\oplus_{i=1}^n \Z^2$ for some $n\geq 4$ and every 
		mapping class extends over $Q$ (and hence over $Q^\dagger$) in such a way that is isotopic to 	the identity on $\partial C$.
\item the manifolds $X_0:= C\bigcup_{id} Q$ and $X_1 = C\bigcup_f Q $ are simply-connected and $(X_0,X_1)$
	is an   exotic pair.
\end{enumerate}
\end{theorem}
This theorem follows from the proof of \cite[Thm. A]{AkbulutRubermanAbsolutelyExotic}. We remark
that $(X_0,X_1)$ is an \emph{absolutely } exotic pair, meaning that there is no diffeomorphism $X_0\to X_1$, in other words we are not fixing the behaviour on the boundary.
\begin{remark}
In \cite{AkbulutRubermanAbsolutelyExotic} it is not explicitly said that the inverse of the cobordism $Q$ is given by $Q^{\dagger}$
but this follows from the fact that the concordance $C$ from the unknot to the knot $J=11n42$ constructed in \cite[Prop. 2.6]{AkbulutRubermanAbsolutelyExotic} is invertible by considering $C^\dagger$, indeed the double of the the slice disk inducing $C$
is an unknotted sphere in $\SS^4$ \cite[Prop. 2.6]{AkbulutRubermanAbsolutelyExotic}.
\end{remark}

The following will be our key proposition.
\begin{proposition}\label{prop:cplCorkAK} Let $(C,f)$ be a cork and let $(X_0,X_1)$ be the simply-connected exotic pair of manifolds with boundary 
obtained by applying the Akbulut-Ruberman construction. 
Then $ \Complexity (C,f) = \Complexity(X_0, X_1)$.
\end{proposition}
\begin{proof} Using the notation of Theorem~\ref{AkbulutRuberman}, $X_0 = C\bigcup_{id} Q$ and $X_1 = C\bigcup_f Q$.
	Let $\Delta$ be an handle decomposition of an $h$-cobordism realizing $\Complexity(X_0, X_1)$
	and let $(P_0, P_1, id)$ be the abstract protocork  associated with $\Delta$.
	As noticed in Remark~\ref{rem:PcorkTwist}, $P_0$ embeds in $X_0 $ and there is a  diffeomorphism
	$\tilde F_1: (X_0\setminus \interior{P_0}) \bigcup_{id}  P_1 \to  X_1$.
	Thanks to  Theorem~\ref{AkbulutRuberman}~\ref{AkbulutRubermanb} we can extend $\tilde F_1$ over $Q^\dagger$ obtaining a map
	\begin{equation*}
		F_1:  (X_0\setminus \interior{P_0}) \bigcup_{id}  P_1 \bigcup_{id} Q^\dagger \to  X_1\bigcup_{id} Q^\dagger
	\end{equation*}
	which restricts to the identity over the boundary $\partial C$.
	Let $\Phi: Q \cup_{id} Q^\dagger \to \partial C \times [0,1]$ be the diffeomorphism fixing the boundary that comes from 
	Theorem~\ref{AkbulutRuberman}~\ref{AkbulutRubermana}.
	We define $F_2 = id \bigcup \Phi$	 as the diffeomorphism 
	\begin{equation*}
		F_2: C\bigcup_{f} Q \bigcup_{id} Q^\dagger \to C\bigcup_{f}\partial C \times I
	\end{equation*}
	obtained by extending $\Phi$ as the identity over $C$.
	Define $F_3: C\bigcup_{f}\partial C \times I\to C\bigcup_{id}\partial C \times I$
	as the diffeomorphism  equal to $f\times id_I$ on the collar $\partial C \times I$ and equal to the identity on $C$.

	We have a diffeomorphism $G = id_C \bigcup \Phi$ between $C\bigcup_{id} \partial C \times I$ and $X_0\bigcup_{id} Q^\dagger$.
	%We use $G$ to  embed  $P_0$  in $C\bigcup_{id} \partial C \times I$,	and we denote this embedded copy by $\tilde{P_0}$.
	The diffeomorphism $G$ induces a diffeomorphism $F_0$,
	\begin{equation*}
		F_0: \left((C\bigcup_{id} \partial C \times I)\setminus \interior{G^{-1}({P_0})} \right) \bigcup_{G^{-1}} P_1 
		\to (X_0\setminus \interior{P_0}) \bigcup_{id} P_1\bigcup_{id} Q^\dagger,
	\end{equation*}
	restricting to the identity on the boundary.
		
	\begin{figure}
	\includegraphics[scale=0.7]{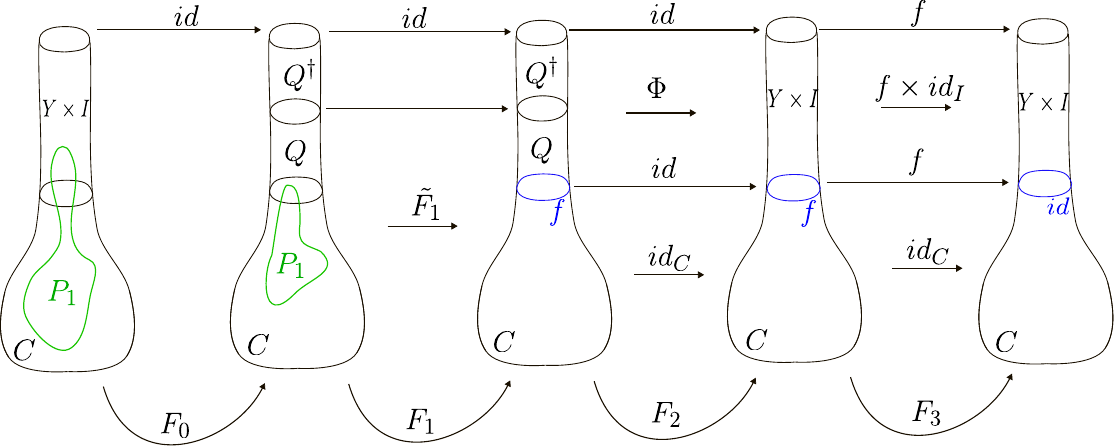}
	\caption{ \label{PicMaps} Composition of the diffeomorphisms used in the proof of Proposition~\ref{prop:cplCorkAK}. 
	For $i\geq 1$, the maps $F_i$ are obtained by 
	stacking diffeomorphisms of cobordisms as showed in the picture.
	The map $F_0$ is equal to  $G$ over the complement of $P_1$ and is equal to  the identity $P_1\to P_1$ over~$P_1$. }

\end{figure}
	
	The composition $F_3 \circ F_2\circ F_1\circ F_0$ which restricts to $f$ on the boundary (this is shown in  Figure~\ref{PicMaps}), 
	shows that the protocork $(P_0, P_1, id)$ supports $(C,f)$, hence by Lemma~\ref{lem:cplCorkSupp}
	\begin{equation*}
		\Complexity(C,f)\leq \Complexity(\Delta) = \Complexity (X_0,X_1).
	\end{equation*}
	Since $(C,f)$ relates $(X_0,X_1)$ by construction, also the reversed inequality holds by Lemma~\ref{lem:CorkRelExoticPair}.
\end{proof}

We are finally ready to prove Theorem~\ref{thm2}.
\begin{proof}[Proof of Theorem~\ref{thm2}.]
In \cite{MorganSzabo99}, Morgan and Szab\'o  construct an infinite family  of inertial $h$-cobordisms $Z_n: X_{n}\to X_{n}$ $n\in \N$,
such that $\Complexity (Z_n) \to \infty$ as $n\to \infty$.
Applying the cork decomposition theorem \cite{CHFS, Matveyev} to $Z_n$, we obtain a family of corks $(C_n,f_n)$ realizing $\Complexity (Z_n) $, hence $\Complexity (C_n, f_n) \to \infty$ as $n\to \infty$.

Let $(X_{n,0}, X_{n,1})$ be the simply-connected exotic pairs obtained by applying the Akbulut-Ruberman construction to $(C_n, f_n)$.
Then by Proposition~\ref{prop:cplCorkAK}, we see that $\Complexity(X_{n,0}, X_{n,1})\to \infty$ as $n\to \infty$.

Consequently, it  is not possible to find a cork relating $(X_{n,0}, X_{n,1})$ for any $n\in \N$.

\end{proof}
\bibliographystyle{alpha}
\bibliography{pcorkPaper}
	
\paragraph{}
\textsc{Max Planck Institut f\"{u}r Mathematik, Bonn, Germany}\newline
\phantom{aa} E-mail address:  \texttt{ladu@mpim-bonn.mpg-de}

\end{document}